\documentclass[12pt,leqno]{amsart}
\usepackage{berry_esseen_mds}

\title{}

\begin{document}
\author{\normalsize\textsc{Denis Kojevnikov}\textsuperscript{\textasteriskcentered}}
\address{\textsuperscript{\textasteriskcentered}\normalfont{Corresponding author. Department of Econometrics and Operations Research, Tilburg University, The Netherlands. Email: \href{mailto:D.Kojevnikov@tilburguniversity.edu}{D.Kojevnikov@tilburguniversity.edu}.}}
\author{\normalsize\textsc{Kyungchul Song}\textsuperscript{\S}}
\address{\noindent\textsuperscript{\S}\normalfont{Vancouver School of Economics, University of British Columbia, Canada.}}

\date{\today}

\begin{center}
	\Large \textsc{A Berry-Esseen Bound for Vector-valued Martingales}
\end{center}

\begin{abstract}
	{\footnotesize This note provides a conditional Berry-Esseen bound for the sum of a martingale difference sequence $\{X_i\}_{i=1}^n$ in $\R^d$, $d\ge 1$, adapted to a filtration $\{\F_i\}_{i=1}^n$. We approximate the conditional distribution of $S=\sum_{i=1}^n X_i$ given a sub-$\sigma$-field $\F_0\subset \F_1$ by that of a mean zero normal random vector having the same conditional variance given $\F_0$ as the vector $S$. Assuming that the conditional variances $\E[X_iX_i^{\top}\mid \F_{i-1}]$, $i\ge 1$, are $\F_0$-measurable and non-singular, and the third conditional moments of $\|X_i\|$, $i\ge 1$, given $\F_0$ are uniformly bounded, we present a simple bound on the conditional Kolmogorov distance between $S$ and its approximation given $\F_0$ which is of order $O_{a.s.}([\ln(ed)]^{5/4}n^{-1/4})$.
	}

	\bigskip

	{\footnotesize \noindent \textsc{Keywords.} Berry-Esseen bound; Gaussian approximation; Martingale-difference sequence; Vector-valued martingale}
\end{abstract}

{
	\let\MakeUppercase\relax
	\maketitle
}

\section{Introduction}
Let $(\Omega,\F,\PM)$ be a probability space and let $\{X_i\}_{i=1}^n$ be an $\R^d$-valued martingale difference sequence with $d\ge 1$ adapted to a filtration $\{\F_i\}_{i=1}^n$, i.e., each $X_i$ is $\F_i$-measurable and $\E[X_{i+1}\mid \F_i]=0$ a.s. In addition, suppose that we are given a sub-$\sigma$-field $\F_0\subset\F_1$, not necessarily trivial, such that $\mathsf{E}[X_1\mid \mathcal{F}_0]=0$ a.s. Throughout the paper we assume that each $X_i$ has finite conditional third moment given $\F_0$, i.e., $\E[\norm{X_i}_{\infty}^3\mid \F_0]<\infty$ a.s., where $\norm{\csdot}_{\infty}$ denotes the maximum norm on $\R^d$.

The goal of this paper is to establish a uniform distributional approximation of the random vector $S\eqdef\sum_{i=1}^n X_i$ conditionally on $\F_0$ by a suitably chosen Gaussian analog. Specifically, we consider a random vector $T$ whose conditional distribution given $\F_0$ is $\ND{0}{V}$, where the covariance matrix $V$ is a version of $\E[SS^\top\mid \F_0]$. Namely, the conditional characteristic function of $T$ is given by
\[
	\E[e^{it^{\top}T}\mid \F_0]=\exp\left(-\frac{1}{2}t^\top V t\right) \qtext{a.s.}
\]
for all $t\in \R^d$. Then we establish a bound on the conditional Kolmogorov distance between $S$ and $T$ given $\F_0$.

Let $\mathcal{A}$ denote the collection of sets of the form $\prod_{j=1}^d(-\infty,r_j]$ with $r\equiv [r_1,\ldots,r_d]^{\top}\in\R^d$. Also, let $\mu_X^{\G}$ denote the regular conditional distribution of a vector $X$ given a sub-$\sigma$-field $\G\subset\F$.\footnote{
	The regular conditional distribution $\mu_Z^{\G}$ of a random vector $Z\in\R^d$ given $\G\subset\F$ satisfies: (i) $\forall B\in\mathcal{B}(\R^d)$, $\mu_Z^{\G}(\csdot,B)$ is a version of $\PR{Z\in B\mid \G}(\csdot)$, and (ii) $\forall\omega\in \Omega$, $\mu_Z^{\G}(\omega,\csdot)$ is a distribution on $\R^d$. In particular, condition (ii) implies that $\DK{X,Y\mid \G}$ defined in \eqref{def:ckd} is $\G$-measurable.
}
The conditional Kolmogorov distance between random vectors $X$ and $Y$ in $\R^d$ given a sub-$\sigma$-field $\G\subset \F$ is defined by
\begin{align}
\label{def:ckd}
	\DK{X,Y\mid \G}(\omega)\eqdef \sup_{A\in\mathcal{A}}\,\abs{\mu_X^{\G}(\omega,A)-\mu_Y^{\G}(\omega,A)}.
\end{align}

Assuming that the conditional variances $\E[X_iX_i^{\top}\mid \F_{i-1}]$, $i\ge 1$, are $\F_0$-measurable, and the third conditional moments of $\|X_1\|_{\infty},\|X_2\|_{\infty},\ldots$ given $\F_0$ are uniformly bounded, we present a simple bound on $\DK{S,T\mid \F_0}$ of order $O_{a.s.}([\ln(ed)]^{5/4}n^{-1/4})$. In addition, we require that the minimum eigenvalues of $\E[X_iX_i^{\top}\mid \F_0]$, $i\ge 1$, are bounded away from zero, that is, the random vectors $X_1,X_2,\ldots$ are assumed to have non-degenerate conditional distributions given $\F_0$.

For scalar-valued martingale difference sequences with constant conditional variances and finite third moments, \cite{Grams:72} showed that $\DK{S,T}=O(n^{-1/4})$. If, in addition, $X_i$, $i\ge 1$, are uniformly bounded, \cite{Bolthausen:82} established a bound of order $O([\ln{n}]n^{-1/2})$. Furthermore, he provided examples of martingale difference sequences for which both estimates are sharp. The classical rate of $O(n^{-1/2})$ is nevertheless possible under stronger conditions on the conditional moments of $X_i$'s. See, for example, \cite{Kiryanova/Rotar:92}, \cite{Renz:96}, and \cite{Wu:20} for recent developments.

In multidimensional settings, extensive research has been focused on sequences of independent random vectors. \cite{Chernozhukov/Chetverikov/Kato:13} established a Berry-Esseen bound of order $O([\ln(dn)]^{7/8}n^{-1/8})$ for maxima of sums of such vectors. This result was subsequently improved in \cite{Chernozhukov/Chetverikov/Kato:17} and \cite{Chernozhukov/Chetverikov/Kato/Koike:19}. Recently, \cite{Lopes:20} provided a nearly $1/\sqrt{n}$ bound on $\DK{S,T}$ for i.i.d.\ sub-Gaussian random vectors, and \cite{Kuchibhotla/Rinaldo:20} improved that result by showing an $O([\ln(en)]^{1/2}n^{-1/2})$ rate of convergence under the weakest possible conditions. This paper relies on the smoothing inequality presented in the latter work.

\section{Preliminary Results}

Let $Z_1,\ldots,Z_n$ be i.i.d.\ standard normal random vectors in $\R^d$ independent of $\mathcal{F}_n$. For $1\le i\le n$, let $Y_i=\Sigma_i^{1/2}Z_i$, where $\Sigma_i$ is a version of $\E[X_iX_i^{\top}\mid \F_0]$. It is clear that the conditional distribution of $T$ given $\F_0$ is the same as that of $\sum_{i=1}^n Y_i$, and so we associate $T$ with the latter sum. In addition, by the properties of conditional distributions,
\begin{equation}
\label{eq:Q_approx}
	\DK{S,T\mid \F_0}=\sup_{r\in \Q^d}\abs{\PR{S\in A_r\mid \F_0}-\PR{T\in A_r\mid \F_0}} \qtext{a.s.},
\end{equation}
where $A_r\eqdef \prod_{j=1}^d(-\infty, r_j]$ with $r\in \R^d$ is a generic set in $\mathcal{A}$, and $\Q$ is the set of rational numbers. 

Consider a random vector $\eta\sim \ND{0}{I_d}$, independent of $Z_1,\ldots,Z_n$ and $\F_n$. We approximate the probabilities on the right-hand side of \eqref{eq:Q_approx} with conditional expectations of the following smooth function:
\[
	\varphi_r(x,\epsilon)\eqdef \PR{x+\epsilon \eta\in A_r},
\]
evaluated at $(S,\varepsilon)$ and $(T,\varepsilon)$, respectively, where $\varepsilon$ is a positive, $\F_0$-measurable random variable which will be determined later. Note that for a fixed $\epsilon>0$, the function $x\mapsto\varphi_r(x,\epsilon)$ is infinitely differentiable, and by Lemma 2.3 in \cite{Fang/Koike:21} for each $x,r\in\R^d$ and $s\ge 1$ we have
\begin{equation}
\label{eq:sd_bound}
	\sum_{j_1,\ldots,j_s=1}^d \abs{\frac{\partial}{\partial x_{j_1}}\cdots\frac{\partial}{\partial x_{j_s}} \varphi_r(x,\epsilon)} \le C_s \epsilon^{-s}[\lnp{d}]^{s/2},
\end{equation}
where $C_s>0$ is a constant depending only on $s$ and $\lnp{x}\equiv 1\vee \ln{x}$. In addition, for an $\F_0$-measurable random variable $\varepsilon$,
\begin{align*}
	&\E[\varphi_r(S,\varepsilon)-\varphi_r(T,\varepsilon)\mid \F_0] \\
	&\qquad=\PR{S+\varepsilon \eta\in A_r\mid \F_0}-\PR{T+\varepsilon \eta\in A_r\mid \F_0} \qtext{a.s.}
\end{align*}

The following lemma establishes an upper bound on the approximation error due to the use of $\varphi_r$. We define
\begin{equation}
\label{eq:min_std}
	\ubar{\sigma}^2\eqdef\min_{1\le j\le d}[V]_{jj}.
\end{equation}
\begin{lemma}
\label{lemma:dk_bound}
Suppose that $\ubar{\sigma}>0$ a.s. There exists a universal constant $C>0$ such that for any $\epsilon>0$,
\[
	\DK{S,T\mid \F_0}\le \sup_{r\in\Q^d}\abs{\E[\varphi_r(S,\epsilon)-\varphi_r(T,\epsilon)\mid \F_0]}+\frac{C\epsilon\lnp{d}}{\ubar{\sigma}} \qtext{a.s.}
\]
\end{lemma}
\begin{proof}
Let $\gamma_{\epsilon}$ denote a mean zero Gaussian measure on $\R^d$ with covariance matrix $\epsilon^2 I_d$. By Lemma 1 in \cite{Kuchibhotla/Rinaldo:20}, for any $r\in\R^d$ and $\epsilon>0$,
\begin{align*}
	&\abs{\left(\mu_S^{\F_0}-\mu_T^{\F_0}\right)(\omega, A_r)} \\
	&\qquad\le \sup_{r\in \Q^d}\abs{\left(\mu_S^{\F_0}\ast\gamma_{\epsilon}-\mu_T^{\F_0}\ast\gamma_{\epsilon}\right)(\omega, A_r)}+\frac{C\epsilon\lnp{d}}{\ubar{\sigma}(\omega)}
\end{align*}
for some universal constant $C>0$. On the other hand, for almost all $\omega\in\Omega$,
\begin{align*}
	&\mathsf{P}(S+\varepsilon \eta\in A_r\mid \F_0)(\omega)-\mathsf{P}(T+\varepsilon \eta\in A_r\mid \F_0)(\omega) \\
	&\qquad=\int 1_{A_r}(x+\epsilon z)\left(\mu_S^{\F_0}\otimes\mu_{\eta}-\mu_T^{\F_0}\otimes \mu_{\eta}\right)(\omega,d (x\times z)) \\
	&\qquad=\left(\mu_S^{\F_0}\ast\gamma_{\epsilon}-\mu_T^{\F_0}\ast\gamma_{\epsilon}\right)(\omega,A_r). \qedhere
\end{align*}
\end{proof}

The next result implies the regularity of the conditional Kolmogorov distance in the sense that for suitable random vectors $X$, $Y$, and $Z$, $\DK{X+Z,Y+Z\mid \F_0}\le\DK{X,Y\mid \F_0}$ a.s.\ when $Z$ is conditionally independent of $X$ and $Y$ given $\F_0$.

\begin{lemma}
\label{lemma:smoothing}
Let $X$, $Y$, and $Z$ be random vectors in $\R^d$ defined on $(\Omega,\F,\PM)$ such that $Z$ is conditionally independent of $X$ and $Y$ given $\F_0$. Then for any $A\in \mathcal{A}$,
\[
	\abs{\PR{X+Z\in A\mid \F_0}-\PR{Y+Z\in A\mid \F_0}}\le \DK{X,Y\mid\F_0} \qtext{a.s.}
\]
\end{lemma}
\begin{proof}
Let $\G\eqdef\F_0\vee \sigma(Z)$. Then
\begin{align*}
	&\PR{X+Z\in A\mid \F_0}-\PR{Y+Z\in A\mid \F_0}\\
	&\qquad=\E[\PR{X+Z\in A\mid \G}-\PR{Y+Z\in A\mid \G}\mid \F_0] \qtext{a.s.},
\end{align*}
and for almost all $\omega\in\Omega$,
\begin{align*}
	&\abs{\PR{X+Z\in A\mid \G}(\omega)-\PR{Y+Z\in A\mid \G}(\omega)}\\
	&\qquad=\abs{\int 1_A(x+Z(\omega)) \left(\mu_X^{\F_0}-\mu_Y^{\F_0}\right)(\omega,dx)}\le \DK{X,Y\mid \F_0}(\omega). \qedhere
\end{align*}

\end{proof}

Finally, we give an upper bound on the moments of the maximum norm of a Gaussian random vector.

\begin{lemma}
\label{lemma:norm_bound}
Let $Y\equiv[Y_1,\ldots,Y_d]^{\top}$ be a zero-mean Gaussian vector in $\R^d$, $d\ge 1$, with $\sigma_j^2\eqdef \E Y_j^2>0$ for all $1\le j\le d$, and let $\bar{\sigma}\eqdef \max_{1\le j\le d}\sigma_j$. Then for any $s\ge 2$,
\begin{equation}
\label{eq:norm_bound}
	\E\norm{Y}_\infty^s\le C_s\bar{\sigma}^s\left(\lnp{d}\right)^{s/2},
\end{equation}
where $C_s>0$ is a constant depending only on $s$.
\end{lemma}
\begin{proof}
Let $f:[a,\infty)\to \R$, $a\ge 0$, be a strictly increasing convex function. Using Jensen's inequality, we have
\[
	\E\norm{Y}_\infty^s\le \E[a\vee \norm{Y}_\infty^s]\le f^{-1}\left(\E[f(a\vee \norm{Y}_\infty^s)]\right).
\]

First, for $s>2$ consider $f(x)=\exp{\left(c_s(x/a)^{2/s}\right)}$ with $a>0$ and $c_s\eqdef s/2-1$, which is convex on $[a,\infty)$. Letting $a=\left(2\sqrt{c_s}\bar{\sigma}\right)^s$, we find that
\begin{align*}
	\E[f(a\vee \norm{Y}_\infty^s)]&=\E \exp\left(c_s\left(1\vee \frac{\norm{Y}_{\infty}^2}{a^{2/s}}\right)\right)\le e^{c_s}\E \exp\left(\frac{\norm{Y}_{\infty}^2}{4\bar{\sigma}^2}\right) \\
	&\le e^{c_s}\sum_{j=1}^p\E\exp\left(\frac{\abs{Y_j}^2}{4\bar{\sigma}^2}\right)=e^{c_s}\sum_{j=1}^d \sqrt{\frac{2\bar{\sigma}^2}{2\bar{\sigma}^2-\sigma_j^2}}\le \sqrt{2}e^{c_s}d,
\end{align*}
and, therefore,
\begin{equation}
\label{eq:norm_bound2}
	\E\norm{Y}_\infty^s\le \left[\ln\left(\sqrt{2}e^{c_s}d \right)\right]^{s/2}(2\bar{\sigma})^s \le  C_s\bar{\sigma}^s\left(\lnp{d}\right)^{s/2}
\end{equation}
for some $C_s>0$ depending only on $s$. For $s=2$ we take $f(x)=\exp(x/(2\bar{\sigma})^2)$ and $a=0$ which similarly yield \eqref{eq:norm_bound2}.
\end{proof}

\section{Main Results}

In this section we derive a Berry-Esseen bound for the random vector $S$. Let $\norm{\csdot}_{e,p}$ denote the element-wise $p$-norm in $\R^{k \times l}$, i.e., for a $k\times l$ matrix $A$, $\norm{A}_{e,p}=\norm{\vecm(A)}_{p}$, $p\in [1,\infty]$, and let
\[
	\ubar{\lambda}^2\eqdef\min_{1\le i\le n}\lambda_{\text{min}}(\Sigma_i),
\]
where $\lambda_{\text{min}}(A)$ is the smallest eigenvalue of $A$.

\begin{lemma}
\label{lemma:approx_cb3}
Suppose that $\ubar{\lambda}>0$ a.s. There exists a universal constant $C>0$ such that for any $\epsilon>0$,
\begin{align*}
	&\sup_{r\in \Q}\abs{\E[\varphi_r(S,\epsilon)-\varphi_r(T,\epsilon)\mid \F_0]} \\
	&\qquad\le C[\lnp{d}]^{3/2}\ubar{\lambda}(\gamma_1+\gamma_3)\epsilon^{-1}+C[\lnp{d}]\beta\ln\left(1+n\ubar{\lambda}^2\epsilon^{-2}\right) \qtext{a.s.},
\end{align*}
where
\begin{align*}
	\gamma_s&\eqdef\max_{1\le i\le n}\left(\E[\|X_i\|_{\infty}^s\mid \F_0]+\bar{\sigma}_i^s[\lnp{d}]^{s/2}\right)/\ubar{\lambda}^s, \quad s>0,\\
	\beta&\eqdef\max_{1\le i\le n}\E\left[\normin{\E[X_iX_i^{\top}\mid \F_{i-1}]-\Sigma_i}_{e,\infty}\mid \F_0\right]/\ubar{\lambda}^2,
\end{align*}
and $\bar{\sigma}_i^2\eqdef\max_{1\le j\le d}[\Sigma_i]_{jj}$.
\end{lemma}

\begin{proof}
First, letting
\[
	U_i\eqdef \sum_{j=1}^{i-1} X_j+\sum_{j=i+1}^n Y_j,
\]
$1\le i\le n$, we write
\begin{align}
\label{eq:interp}
	\begin{aligned}
		&\abs{\E[\varphi_r(S,\epsilon)-\varphi_r(T,\epsilon)\mid \F_0]} \\
		&\qquad\le \sum_{i=1}^{n} \abs{\E[\varphi_r(U_i+X_i,\epsilon)-\varphi_r(U_i+Y_i,\epsilon)\mid \F_0]}  \qtext{a.s.}
	\end{aligned}
\end{align}
Consider the right hand side of the preceding display. For each $1\le i\le n$, let $S_i=S_{i-1}+X_i$ and $T_i=T_{i-1}+Y_i$ with $S_0\equiv 0$ and $T_0\equiv 0$. We also define
\[
	\varepsilon_i\eqdef \left(\epsilon^2+(n-i)\ubar{\lambda}^2\right)^{1/2} \qtext{and}\quad V_i\eqdef\left(\sum_{k=i+1}^n \Sigma_k-(n-i)\ubar{\lambda}^2 I_d\right)^{1/2}.
\]
Since $Y_1,\ldots, Y_n$ are conditionally independent of $\F_n$ given $\F_0$, by Lemma \ref{lemma:smoothing} we have
\begin{align*}
	&\abs{\E[\varphi_r(U_i+X_i,\epsilon)-\varphi_r(U_i+Y_i,\epsilon)\mid \F_0]} \\
	&\qquad=\abs{\PR{S_{i-1}+X_i+\varepsilon_i \eta\in A_{r - V_i\eta'}\mid \F_0}-\PR{S_{i-1}+Y_i+\varepsilon_i \eta\in A_{r - V_i\eta'}\mid \F_0}} \\
	&\qquad\le \sup_{r\in\Q^d}\abs{\E[\varphi_r(S_{i-1}+X_i,\varepsilon_i)-\varphi_r(S_{i-1}+Y_i,\varepsilon_i)\mid \F_0]} \qtext{a.s.}
\end{align*}
for each $1\le i< n$, where $\eta'$ is an independent copy of $\eta$.

\begin{claim}\label{lemma:taylor_bounds}
There exists a universal constant $C>0$ such that for each $r\in\R$,
\begin{align}
	\label{eq:taylor_bound1}
	\begin{aligned}
		&\E[\varphi_r(S_{i-1}+X_i,\varepsilon_i)-\varphi_r(S_{i-1}+Y_i,\varepsilon_i)\mid \F_0] \\
		&\qquad\le C\varepsilon_i^{-2}[\lnp{d}]\ubar{\lambda}^2\beta+ C\varepsilon_i^{-3}[\lnp{d}]^{3/2}\ubar{\lambda}^3\gamma_3,
	\end{aligned}
\end{align}
if $1\le i< n$, and
\begin{align}
	\label{eq:taylor_bound2}
	\begin{aligned}
		&\E[\varphi_r(S_{n-1}+X_n,\epsilon)-\varphi_r(S_{n-1}+Y_n,\epsilon)\mid \F_0] \\
		&\qquad\le C\epsilon^{-1}[\lnp{d}]^{1/2}\ubar{\lambda}\gamma_1.
	\end{aligned}
\end{align}
\end{claim}

\begin{subproof}
We show \eqref{eq:taylor_bound1}. The inequality \eqref{eq:taylor_bound2} follows using similar arguments. Let $h_{1i}(\tau)\eqdef \varphi_r(S_{i-1}+\tau X_i,\varepsilon_i)$ and $h_{2i}(\tau)\eqdef \varphi_r(S_{i-1}+\tau Y_i,\varepsilon_i)$. Using Taylor's expansion up to terms of the third order,
\begin{align*}
	h_{1i}(1)-h_{2i}(1)=\sum_{j=1}^2 \frac{1}{j!}\left(h_{1i}^{(j)}(0)-h_{2i}^{(j)}(0)\right)+\frac{1}{3!}\left(h_{1i}^{(3)}(\tau_1)-h_{2i}^{(3)}(\tau_2)\right),
\end{align*}
where $\abs{\tau_1},\abs{\tau_2}\le 1$. First, it is clear that
\[
	\E[\E[h_{1i}'(0)-h_{2i}'(0)\mid \F_{i-1}]\mid \F_0]=0 \qtext{a.s.},
\]
and, using \eqref{eq:sd_bound},
\begin{align*}
	&\abs{\E[h_{1i}''(0)-h_{2i}''(0)\mid \F_0]}\le \E[\abs{\E[h_{1i}''(0)-h_{2i}''(0)\mid \F_{i-1}]}\mid \F_0] \\
	&\qquad \le C'\varepsilon_i^{-2}[\lnp{d}]\E[\normin{\E[X_iX_i^{\top}\mid \F_{i-1}]-\Sigma_i}_{e,\infty}\mid \F_0] \qtext{a.s.},
\end{align*}
where $C'>0$ is a universal constant. Finally, using \eqref{eq:sd_bound} and Lemma \ref{lemma:norm_bound},
\begin{align*}
	\absin{\E[h_{1i}^{(3)}(\tau_1)-h_{2i}^{(3)}(\tau_2)\mid \F_0]}&\le \E[\absin{h_{1i}^{(3)}(\tau_1)}\mid \F_0]+\E[\absin{h_{2i}^{(3)}(\tau_2)}\mid \F_0] \\
	&\le C''\varepsilon_i^{-3}[\lnp{d}]^{3/2}\left(\E[\|X_i\|_{\infty}^3\mid \F_0]+\bar{\sigma}_i^3[\lnp{d}]^{3/2}\right) \qtext{a.s.},
\end{align*}
where $C''>0$ is a universal constant.
\end{subproof}

Using Claim \ref{lemma:taylor_bounds}, the result follows from \eqref{eq:interp} by noticing that
\[
	\sum_{i=1}^{n-1}\varepsilon_i^{-2}\le \int_0^1\frac{n-1}{(\epsilon^2+(n-1)\ubar{\lambda}^2x)}\, dx\le \frac{1}{\ubar{\lambda}^2}\ln\left(1+\frac{n\ubar{\lambda}^2}{\epsilon^2}\right)
\]
and
\[
	\sum_{i=1}^{n-1}\varepsilon_i^{-3}\le \int_0^1\frac{n-1}{(\epsilon^2+(n-1)\ubar{\lambda}^2x)^{3/2}}\, dx\le \frac{2}{\ubar{\lambda}^2\epsilon}. \qedhere
\]
\end{proof}

\medskip

\begin{theorem}
\label{thm:BE_bound}
Suppose that $\ubar{\lambda}>0$ a.s. There exists a universal constant $C>0$ such that
\begin{align}
\label{eq:BE_bound}
	\begin{aligned}
		\DK{S,T\mid \F_0}&\le C[\lnp{d}]^{5/4}\left(\gamma\ubar{\lambda}/\ubar{\sigma}\right)^{1/2} \\
		&\quad+C[\lnp{d}]\beta\ln\left(1+\frac{n\ubar{\lambda}/\ubar{\sigma}}{[\lnp{d}]^{1/2}\gamma}\right) \qtext{a.s.},
	\end{aligned}
\end{align}
where $\gamma\equiv \gamma_1+\gamma_3$, and $\ubar{\sigma}$ is defined in \eqref{eq:min_std}.
\end{theorem}

\begin{proof}
Using Lemmas \ref{lemma:dk_bound} and \ref{lemma:approx_cb3}, we find that for any $\epsilon>0$,
\begin{align*}
	\DK{S,T\mid \F_0}&\le \frac{C[\lnp{d}]^{3/2}\ubar{\lambda}\gamma}{\epsilon} \\
	&\quad+C[\lnp{d}]\beta\ln\left(1+\frac{n\ubar{\lambda}^2}{\epsilon^2}\right)+\frac{C\epsilon \lnp{d}}{\ubar{\sigma}} \qtext{a.s.},
\end{align*}
where $C$ is a universal constant. Since this inequality holds for all $\epsilon>0$, it also holds for random $\epsilon$ a.s.\ on the event $\{\epsilon\in (0,\infty)\}$. Consequently, the result follows by choosing $\epsilon=[\lnp{d}]^{1/4}(\ubar{\lambda}\ubar{\sigma}\gamma)^{1/2}$ and noticing that $\ubar{\sigma}\ge \ubar{\lambda}$.
\end{proof}

\begin{remark*}
(1) If the conditional variances $\E[X_iX_i^{\top}\mid \F_{i-1}]$, $1<i\le n$, are $\F_0$-measurable, then $\beta=0$ a.s., and the bound in Theorem \ref{thm:BE_bound} becomes
\begin{align*}
	\DK{S,T\mid \F_0}&\le C[\lnp{d}]^{5/4}\left(\gamma\ubar{\lambda}/\ubar{\sigma}\right)^{1/2} \\
	&\le  C \gamma^{1/2} [\lnp{d}]^{5/4} n^{-1/4} \qtext{a.s.}
\end{align*}
because 
\[
	\ubar{\sigma}^2/n \ge \min_{1\le i\le n}\min_{1\le j\le d}[\Sigma_i]_{jj} \ge \ubar{\lambda}^2 \quad \qtext{a.s.}
\]
In this case, when $\sup_{i\ge 1}\E[\norm{X_i}_{\infty}^3\mid \F_0]<\infty$ a.s.,  and the smallest eigenvalues of $\Sigma_1,\Sigma_2,\ldots$ are uniformly bounded away from zero, the bound is of order $O_{a.s.}([\lnp{d}]^{5/4} n^{-1/4})$.

\noindent(2) Noticing that $\ln(1+x)\le \sqrt{x}$ for $x\ge 0$, the second term on the right hand side of \eqref{eq:BE_bound} can be further bounded by
\[
	\frac{[\lnp{d}]^{3/4}\sqrt{n}\beta}{\sqrt{\gamma \ubar{\sigma}/\ubar{\lambda}}}.
\]
The latter quantity is similar to the corresponding term of the bound given in Theorem 2 in  Section 9.3 of \cite{Chow/Teicher:97:Prob} for scalar-valued martingales. The corresponding first term is, however, of order $O(n^{-1/8})$ under the conditions of part (1).

\vspace{0.5em}

\noindent(3) The bound in \eqref{eq:BE_bound} trivially applies to maxima of vector-valued martingales because for $r\in\R$ and a random vector $\xi\in \R^d$, $\{\xi\in A_{r\boldsymbol{i}}\}=\{\max_{1\le j\le d}\xi_j\le r\}$, where $\boldsymbol{i}$ is a vector of ones, and therefore, letting $M(\xi)\eqdef \max_{1\le j\le d}\xi_j$,
\[
	\DK{M(S),M(T)\mid \F_0}\le \DK{S,T\mid \F_0} \qtext{a.s.}
\]
\end{remark*}

\bibliographystyle{elsart-harv}
\bibliography{berry_esseen_mds}

\end{document}